\documentclass[12pt]{article}
\usepackage{fullpage}
\usepackage{amsmath,amsthm,amssymb}
\usepackage{xspace}

\IfFileExists{../EmanueleViolaDir.txt}{\def\EmanueleViolaDir{1}}{
\def\EmanueleViolaDir{0}}

\DeclareMathOperator{\Tr}{Tr}
\DeclareMathOperator{\tr}{Tr}

\ifnum\EmanueleViolaDir=1
\usepackage{../flexiblemathdisplay}
\else
\usepackage{flexiblemathdisplay}
\fi

\newtheorem{theorem}{Theorem}[section]
\newtheorem{conjecture}[theorem]{Conjecture}
\newtheorem{lemma}[theorem]{Lemma}
\newtheorem{corollary}[theorem]{Corollary}

\newtheorem{claim}[theorem]{Claim}

\newtheorem{definition}[theorem]{Definition}
\newtheorem{notation}[theorem]{Notation}

\newtheorem{fact}[theorem]{Fact}

\newcommand{\zo}{\{0, 1\}}

\newcommand{\e}{\epsilon}
\newcommand{\eps}{\epsilon}

\def\E{\mathbb{E}}
\def\P{\mathbb{P}}

\def\a{\alpha}
\def\b{\beta}
\def\g{\gamma}
\def\d{\delta}
\def\G{\Gamma}
\def\D{\Delta}
\def\R{\mathbb{R}}
\def\F{\mathbb{F}}

\def\SL2q{\text{SL}(2,q)}
\def\sl2q{\text{SL}(2,q)}

\newcommand{\inpr}{\bullet}

\renewcommand{\Pr}{\P}

\newcommand{\cc}{\mathrm{Class}}
\newcommand{\uc}{\mathbf{C}}

\begin{document}

\begin{titlepage}

\title{Interleaved group products\thanks{Preliminary versions of this paper have appeared as \cite{GowersV-cc-int,GowersV-cc-int-2}.}}
\author{W.~T.~Gowers\thanks{Royal Society 2010 Anniversary Research Professor.} \and Emanuele Viola\thanks{Supported by NSF grant
CCF-1319206.  Work done in part while a visiting scholar at Harvard University,
with support from Salil Vadhan's Simons Investigator grant, and in part while
visiting the Simons institute for the theory of computing. Email:
\texttt{viola@ccs.neu.edu}.}
}

\maketitle

\begin{abstract}
Let $G$ be the special linear group $\sl2q$.  We show that if
$(a_1,\ldots,a_t)$ and $(b_1,\ldots,b_t)$ are sampled uniformly from large
subsets $A$ and $B$ of $G^t$ then their interleaved product $a_1 b_1
a_2 b_2 \cdots a_t b_t$ is nearly uniform over $G$.  This extends a result
of the first author \cite{Gowers08}, which corresponds to the independent
case where $A$ and $B$ are product sets.  We obtain a number of other
results.  For example, we show that if $X$ is a probability distribution on
$G^m$ such that any two coordinates are uniform in $G^2$, then a
pointwise product of $s$ independent copies of $X$ is nearly uniform in
$G^m$, where $s$ depends on $m$ only. Extensions to other groups are
also discussed.

We obtain closely related results in communication complexity, which is
the setting where some of these questions were first asked by Miles and
Viola \cite{MilesV-leak}.  For example, suppose party $A_i$ of $k$ parties
$A_1,\dots,A_k$ receives on its forehead a $t$-tuple $(a_{i1},\dots,a_{it})$
of elements from $G$. The parties are promised that the interleaved
product $a_{11}\dots a_{k1}a_{12}\dots a_{k2}\dots a_{1t}\dots a_{kt}$ is
equal either to the identity $e$ or to some other fixed element $g\in G$,
and their goal is to determine which of the two the product is equal to. We
show that for all fixed $k$ and all sufficiently large $t$ the communication
is $\Omega(t \log |G|)$, which is tight.  Even for $k=2$ the previous best
lower bound was $\Omega(t)$.  As an application, we establish the
security of the leakage-resilient circuits studied by Miles and Viola \cite{MilesV-leak}
in the ``only computation leaks'' model.
\end{abstract}

\thispagestyle{empty}
\end{titlepage}

\section{Introduction and our results}

Let $G$ be a finite group.  All our results are
asymptotic in the size of the group, so $G$ should be
considered large.  Suppose we have $m$ probability distributions
$X_i$ over $G$, each of high entropy.  For
this discussion, we can think of each $X_i$ as being
uniform over a constant fraction of $G$. We will first
consider the case where the $X_i$ are independent, and
later we will throw in dependencies.

If we sample $x_i$ from $X_i$ and output the product $x_1\dots x_m$,
the resulting distribution $D=\prod_{i\leq m}X_i$ is the convolution
of the distributions $X_i$. Our aim is to show that $D$ closely approximates
the uniform distribution on $G$. More precisely, we ask for the approximation
to be uniform: we would like to show that
\[ | \Pr[D = g] - 1/|G| | \le \eps/|G|,\]
for every $g\in G$. Such a bound guarantees in particular that $D$ is supported
over the entire group. It also immediately implies that $D$ is $\eps$-close to
uniform in statistical distance, that is, in the $\ell_1$ norm.

The above goal has many applications in group theory, see
for example \cite{Gowers08,BabaiNP08} and the citations
therein.  As discussed later, it is also closely related
to problems in communication complexity and cryptography.

As a warm-up, consider the case $m = 2$. Here we have
only two distributions $X$ and $Y$, and it is easy
to see that $X Y$ does not mix, no matter which group is
considered. Indeed, let $X$ be uniform over an arbitrary
subset $S$ of $G$ of density $1/2$, and let $Y$ be
(uniform over) the set of the same density consisting of
all the elements in $G$ except the inverses of the
elements in $S$, i.e., $Y := G \setminus S^{-1}$.  It is
easy to see that $XY$ never equals $1_G$. (In some groups
we get a good $\ell_2$ approximation, so when we say that
$XY$ does not mix we mean in the uniform sense mentioned
above.)

Now consider the case $m=3$, so we have three
distributions $X$, $Y$, and $Z$.  Here the answer depends
on the group $G$.  It is easy to see that if $G$ has a
large non-trivial subgroup $H$ then $D := X Y Z$ does not
mix.  Indeed, we can just let each distribution be
uniform over $H$.  It is also easy to see that $X+Y+Z$ does
not mix over the abelian group $Z_p$.  For example, if
$X=Y=Z$ are uniform over $\{0,1,\ldots,p/4\}$ then
$X+Y+Z$ is never equal to $p-1$.

However, for other groups it is possible to establish a
good bound.  This was shown by Gowers \cite{Gowers08}.
The following version of the result was given by Babai,
Nikolov, and Pyber.  The inequality can be stated in
terms of the $2$-norm of the probability distributions.
In this paper we shall use two different normalizations of
the 2-norm, so to avoid confusion we shall use different
notation for the two.

\begin{notation}
Let $G$ be a finite group and let $v:G\to\mathbb R$. We denote by
$\| v \|_{2}$ the $\ell_2$ norm
$\sqrt{\sum_x v(x)^2}$ of $v$ and by $\|v\|_{L_2}$ the $L_2$
norm $\sqrt{\E_xv(x)^2}$ of $v$. Here $\E$ denotes the average
over~$G$.
\end{notation}

The inequality of Babai, Nikolov, and Pyber is the following.

\begin{theorem}[\cite{BabaiNP08}]
\label{t-gowers-bnp-XY} Let $G$ be a finite group and let $X$ and $Y$
be two independent distributions over $G$ with product $D$. Let $U$
be the uniform distribution over $G$. Then
\[ \|D-U\|_{2} \le \|X\|_{2} \|Y\|_{2}\sqrt{|G|/d},\]
where $d$ is the minimum dimension of a non-trivial
representation of $G$.
\end{theorem}

It is also essentially present in \cite{Gowers08} as Lemma 3.2: there the
inequality is stated (in an equivalent form) in the case where one of the
two distributions is uniform over a subset of $G$, but the proof yields the
more general result with hardly any modification. (Babai, Nikolov and
Pyber give a slightly different argument, however, and it was subsequently
observed by several people that there is a short and natural proof using
non-abelian Fourier analysis.)

From Theorem \ref{t-gowers-bnp-XY} and the Cauchy-Schwarz inequality
one can immediately deduce the following inequality for three distributions,
where a uniform bound is obtained.

\begin{theorem}[\cite{BabaiNP08}]
\label{t-gowers-bnp} Let $G$ be a group, and let $g$ be
an element of $G$. Let $X$, $Y$, and $Z$ be three
independent distributions over $G$.  Then
\[ | \Pr[XYZ = g] - 1/|G| | \le \|X\|_{2} \|Y\|_{2} \|Z\|_{2} \sqrt{|G|/d},\]
where $d$ is the minimum dimension of a non-trivial
representation of $G$.
\end{theorem}

When each distribution is uniform
over a constant fraction of $G$, the right-hand side
becomes
\[ O(d^{-1/2})/|G|.\]
Note that the parameter $\eps$ in our goal above then becomes
$O(d^{-1/2})$. We mention that for any non-abelian simple group we have
$d \ge \sqrt{\log |G|}/2$, whereas for $G$ the special linear group $\sl2q$
we have $d \ge |G|^{1/3}$, cf.~\cite{Gowers08}.  In particular, for $G =
\sl2q$ we have that $XYZ$ is $\eps$-close to uniform over the group,
where $\eps = 1/|G|^{-\Omega(1)}$.  Later we shall also give an alternative
proof of this last bound.

\paragraph{Dependent distributions.}
In this paper we consider the seemingly more difficult case where there
may be dependencies across the $X_i$. To set the scene, consider three
distributions $A$, $Y$, and $A'$, where $A$ and $A'$ may be dependent,
but $Y$ is independent of $(A,A')$. Must the distribution $AYA'$ mix? It
is not hard to see that the answer is no. Indeed, let $Y$ be uniformly
distributed over an arbitrary set $S$ of density $1/2$. We may now define
$(A,A')$ to be uniformly distributed over all pairs $(x,y)$ such that $1_G\notin xSy$.
Then the marginal distributions $A$ and $A'$ are both uniform over the
whole of $G$, but $AYA'$ is never equal to $1_G$.

One of our main results is that mixing does, however, occur for
distributions of the form $A B A' B'$, where $A$ and $A'$
are dependent, and $B$ and $B'$ are also dependent, but
$(A,A')$ and $(B,B')$ are independent.  Moreover, if we look
at interleaved products of longer tuples, then the
bound scales in the desired way with the length $t$ of
the tuples.

\begin{definition} The \emph{interleaved product} $a \inpr b$ of two tuples
$(a_1,a_2,\ldots,a_t)$ and $(b_1,b_2,\ldots,b_t)$ is defined as
\[ a \inpr b := a_1 b_1 a_2 b_2 \cdots a_t b_t.\]
\end{definition}

\begin{theorem} \label{th-main-v1}  Let $G = \sl2q$.
Let $A, B \subseteq G^t$ have densities $\alpha$ and
$\beta$ respectively.  Let $g \in G$.  If $a$ and $b$ are
selected uniformly from $A$ and $B$ we have
\[ | \Pr[ a \inpr b = g] - 1/|G| | \le (\alpha
\beta)^{-1} |G|^{-\Omega(t)} / |G|.\]
\end{theorem}

In particular, the distribution $a \inpr b$ is
at most $(\alpha \beta)^{-1} |G|^{-\Omega(t)}$ away from
uniform in statistical distance. Here $\Omega(t)$ denotes
a function that is bounded below by $ct$ for some $c>0$.

For the case of $t=2$ we obtain a result that applies to
arbitrary distributions and is sharper: the factor $1/\alpha \beta$ is
improved to $\sqrt{1/\alpha \beta}$.

\begin{theorem} \label{t-t2-arbitrary}
Let $G$ be the group SL$(2,q)$.  Let $u$ and $v$ be two independent
distributions over $G^2$. Let $a$ be sampled according to $u$ and $b$
according to $v$. Then, for every $g\in G$, \[ |\Pr_{a,b}[a \inpr b=g] - 1/|G| |
\le \gamma|G| \cdot \|u\|_{2} \|v\|_{2},\]
where $\gamma$ can be taken to be of the form $|G|^{-\Omega(1)}$.
\end{theorem}

To get a feel for what this bound is saying, note that if $u$ and
$v$ are uniform over subsets of $G^2$ of densities $\alpha$ and $\beta$,
respectively, then $\|u\|_{2} = (\alpha
|G|^2)^{-1/2}$ and $\|v\|_{2} = (\beta |G|^2)^{-1/2}$,
and so the upper bound is $(\alpha \beta)^{-1/2}
\gamma/|G|$. Thus, in general we get a good uniform bound provided
that $\alpha\beta$ is significantly greater than $\gamma^2$, so for the
$\gamma$ above we can take $\alpha$ and $\beta$ as small as $|G|^{-\Omega(1)}$.

From Theorem \ref{t-t2-arbitrary} we obtain a number of other results
which we now describe. Call a distribution over $G^m$ \emph{pairwise
uniform} if any two coordinates are uniform in $G^2$.  We show that the
product of a sufficiently large number of pairwise uniform distributions over
$G^m$ is close to uniform over the entire space $G^m$.

\begin{theorem} \label{t-intro-pairwise-uniform}
Let $G=\sl2q$.  For every $m \ge 2$ there exists $r$ such that the
following is true.  Let $\mu_1,\dots,\mu_r$ be pairwise uniform distributions
on $G^m$. Let $\mu$ be the distribution obtained by taking the pointwise
product of random samples from the $\mu_i$.
Then $\mu$ is $1/|G|$ close in statistical distance to the uniform
distribution over $G^m$.
\end{theorem}

As we shall see later, the parameter $1/|G|$ is not too
important in the sense that it can be made smaller by
making $r$ larger.  Note that the assumption that the
distributions are pairwise uniform cannot be relaxed to
the assumption that each coordinate is uniform.  A simple
counterexample is to take each $\mu_i$ to
be uniform on the set of points of the
form $(x,x,\dots,x)$.

As mentioned earlier, our results also imply a special case of Theorem
\ref{t-gowers-bnp}.  Recall that the latter bounds the distance between
$XYZ$ and uniform.  Our Theorem \ref{th-main-v1} with $t=2$ immediately
implies a similar result but with four distributions, i.e., a bound on the
distance of $WXYZ$ from uniform.  To obtain a result about three
distributions like Theorem \ref{t-gowers-bnp} we make a simple and
general observation that mixing in four steps implies mixing in three, see
\S\ref{s-misc}.  Thus we recover, up to polynomial factors, the bound in
Theorem \ref{t-gowers-bnp} for the special case of $G = \sl2q$.  Unlike the
original proofs, ours avoids representation theory.

A more significant benefit of our proof is that it applies to other settings.
For example, in the next theorem we can prove the $XYZ$ result even if
one of the distribution is dense only within a conjugacy class, as opposed
to the whole group.

\begin{theorem} \label{t-XYZ-cc}
Let $G=SL(2,q)$.  For all but $O(1)$ conjugacy classes $S$ of $G$ the
following holds.  Let $A$ be a subset of $S$ of density $|A|/|S| = \alpha$.
Let $B$ and $C$ be subsets of $G$ of densities $|B|/|G|=\beta$ and
$|C|/|G|=\gamma$.  Pick $a$, $b$, and $c$ uniformly and independently
from $A$, $B$, and $C$ respectively.  Let $g \in G$.  Then $|\Pr[abc=g] -
1/|G|| \le (\alpha \beta \gamma)^{-1} |G|^{-\Omega(1)}/|G|.$


\end{theorem}


We show that theorems \ref{th-main-v1}, \ref{t-t2-arbitrary}, and
\ref{t-intro-pairwise-uniform} above, and Theorem \ref{th-main-cc} in
\S\ref{s-intro-cc} hold for any group $G$ that satisfies a certain condition
about conjugacy classes. We then prove that that condition is satisfied for
$\sl2q$. In order to state the condition, we need some notation.

\begin{notation}
Let $G$ be a group.  For $x \in G$ we write $\uc(x)$ for the uniform
distribution on the conjugacy class of $x$, i.e., the distribution $u^{-1} x u$
for uniform $u \in G$.  Different occurrences of $\uc$ correspond to
independent choices of $u$.
\end{notation}

The theorem below states that the condition is satisfied (so the condition is
the conclusion of the theorem).

\begin{theorem} \label{t-sl2q-nice}
Let $G = \sl2q$ and let $a \in G$.  Then
\[\E_{b,b' \in G}\P[\uc(ab^{-1})\uc(b)
= \uc(ab'^{-1})\uc(b')] \le (1+\gamma)/|G|\]
with $\gamma =|G|^{-\Omega(1)}$.
\end{theorem}

Actually for our results we only need this theorem when $a$ is uniform over $G$.
We note that the left-hand side of the inequality above is the collision probability of the distribution
$\uc(ab^{-1})\uc(b)$ for uniform $b \in G$.



We conjectured \cite{GowersV-cc-int,GowersV-cc-int-2} that Theorem
\ref{t-sl2q-nice} holds for all groups of Lie type of bounded rank, and that a
weaker version of Theorem \ref{t-sl2q-nice} holds for all non-abelian simple
groups.  After our work, Shalev confirmed this \cite{Shalev16}.  As a
consequence, theorems \ref{th-main-v1}, \ref{t-t2-arbitrary},
\ref{t-intro-pairwise-uniform}, and \ref{th-main-cc} hold as stated for groups
of Lie type of bounded rank, and weaker versions of theorems
\ref{th-main-v1}, \ref{t-t2-arbitrary}, \ref{t-intro-pairwise-uniform}, and
\ref{th-main-cc} hold for all non-abelian simple groups. This improves on a
result in \cite{GowersV-cc-int} which also applied to all non-abelian simple
groups. For a concrete example consider the alternating group.  For this
group \cite{Shalev16} proves a bound of the same form as that of Theorem
\ref{t-sl2q-nice} but with $\gamma =1/\log^{\Omega(1)} |G|$. This yields
Theorem \ref{t-t2-arbitrary} for the alternating group with this weaker
$\gamma$.  Miles and Viola \cite{MilesV-leak} show that this result is best
possible up to the constant in the $\Omega(1)$, and so the same applies
to the result in \cite{Shalev16} about the alternating group.


\medskip

Our results above are closely related to results in \emph{communication
complexity}, which is the setting in which some of them were originally
asked \cite{MilesV-leak}. We discuss this perspective next.

\subsection{Communication complexity} \label{s-intro-cc}

Computing the product $\prod_{i \le t} g_i$ of a
given tuple $(g_1,\ldots,g_t)$ of elements from a group
$G$ is a fundamental task. This is for two reasons.
First, depending on the group, this task is complete for
various complexity classes
\cite{KrohnMR66,CookM87,Barrington89,Ben-OrC92,ImmermanL95,Miles14}.
For example, Barrington's famous result
\cite{Barrington89} shows that it is complete for NC$^1$
whenever the group is non-solvable, a result which
disproved previous conjectures.  Moreover, the reduction
in this result is very efficient: a projection. The
second reason is that such group products can be randomly
self-reduced \cite{Babai87,Kilian88}, again in a very
efficient way. The combination of completeness and
self-reducibility makes group products extremely
versatile, see
e.g.~\cite{FeigeKN94,ApplebaumIK06,GoldwasserGHKR08,MilesV-leak}.

Still, some basic open questions remain regarding the complexity of
iterated group products.  Here we study a communication complexity
question raised by Miles and Viola in \cite{MilesV-leak}, which was the
starting point of our work.  This question is interesting already in Yao's
basic 2-party communication model \cite{Yao79}.  However we will be able
to answer it even in the multiparty number-on-forehead model
\cite{CFL83}.  So we now describe the latter.  For background, see the
book \cite{KuN97}.  There are $k$ parties $A_1,\dots,A_k$ who wish to
compute the value $f(x_1,\dots,x_k)$ of some function of $k$ variables,
where each $x_i$ belongs to some set $X_i$. The party $A_i$ knows the
values of all the $x_j$ apart from $x_i$ (one can think of $x_i$ as being
written on $A_i$'s forehead).  They write bits on a blackboard according to
some protocol: the \emph{communication complexity} of $f$ is the smallest
number of bits they will need to write in the worst case.

The overlap of information makes proving lower bounds in this model
useful and challenging.  Useful, because such bounds find a striking
variety of applications; see for example the book \cite{KuN97} for some of
the earlier ones. This paper adds to the list an application to cryptography.
Challenging, because obtaining a lower bound even for $k=3$ parties
typically requires different techniques from those that may work for $k=2$
parties.  This is reflected in the sequence of papers
\cite{GowersV-cc-int,GowersV-cc-int-2} leading to the present one.

In this paper we consider the following problem, posed in
\cite{MilesV-leak}. Each $x_i$ is a sequence $(a_{i1},\dots,a_{it})$ of $t$
group elements, and we define their \emph{interleaved product} to be
\[x_1\bullet x_2\bullet\dots\bullet x_k=a_{11}\dots a_{k1}a_{12}\dots a_{k2}\dots a_{1t}\dots a_{kt},\]
which we shall sometimes write as
$\prod_{j=1}^ta_{1j}\dots a_{kj}$.  In other words, the
entire input is a $k\times t$ matrix of elements from
$G$, party $i$ knows all the elements except those in row
$i$, and the interleaved product is the product in column
order.  The parties are told that $x_1\bullet\dots\bullet
x_k$ is equal either to the identity $e$ or to a
specified group element $g$, and their job is to
determine which.

If the group is abelian the problem can be solved with communication
$O(1)$ by just two players, using the public-coin protocol for equality. Over
certain other groups a communication lower bound of $t/2^{O(k)}$ follows
via \cite{Barrington89} from the lower bound in \cite{ChorGo88,BNS92} for
generalized inner product; cf.~\cite{MilesV-leak}. However, this bound
does not improve with the size of the group.  In particular it is far from the
(trivial) upper bound of $O(t \log |G|)$, and it gives nothing when $t =
O(1)$.  We stress that no better results were known even for the case of
$k=2$ parties.

Motivated by a cryptographic application which is
reviewed below, the paper \cite{MilesV-leak} asks whether
a lower bound that grows with the size of the group,
ideally $\Omega(t \log |G|)$, can be established over
some group $G$.

Here we show that if $t \ge b^{2^k}$ where $b$ is a
certain constant, then the communication is at least
$(t/b^{2^k})\log |G|$, even for randomized protocols that
are merely required to offer a small advantage over
random guessing.
In particular, for all fixed $k$ and all sufficiently
large $t$ we obtain an $\Omega(t \log |G|)$ lower bound,
which is tight.

\begin{theorem} \label{th-main-cc}
There is a constant $b$ such that the following holds. Let $G = \sl2q$.  Let
$P : G^{k \times t} \to \zo$ be a $c$-bit $k$-party number-on-forehead
communication protocol.  For $g \in G$ denote by $p_g$ the probability
that $P$ outputs 1 over a uniform input $(a_{i,j})_{i \le k, j \le t}$ such that
$\prod_{j=1}^ta_{1j}\dots a_{kj} = g$.  For any two $g,h \in G$ we have:

(1) For any $k$, if $t \ge b^{2^k}$ then $|p_g - p_h| \le 2^c \cdot
|G|^{-t/b^{2^k}}.$

(2) For $k=2$, if $t \ge 2$ then $|p_g - p_h| \le 2^c \cdot |G|^{-t/b}.$
\end{theorem}

We note that the problem for $t=1$ can be solved with $O(1)$ communication using
the public-coin protocol for equality, cf.~\cite{KuN97}.  The same technique
solves with $O(1)$ communication the variant of the $t=2$ case where one
element, say $a_1$, is fixed to the identity. For $k=2$ parties we show that the
case $t=2$ is hard; for $k>2$ it remains open to determine what is the smallest
$t$ for which the problem is hard.

We conjecture that the doubly-exponential $b^{2^k}$ terms in Theorem
\ref{th-main-cc} can be replaced with the singly exponential $b^k$.  This
would match the state-of-the-art lower bounds \cite{BNS92}.  In fact, we
make a much bolder conjecture.  Let us first review the context.  A central
open problem in number-on-forehead communication complexity is to
prove lower bounds when the number of players is more than logarithmic
in the input length, cf.~\cite{KuN97}.  Moreover, there is a shortage of
candidate hard functions, thanks to the many clever protocols that have
been obtained \cite{Grolmusz94,BGKL03,PRS97,
Ambainis96,AmbainisL00,AdaCFN12,ChattopadhyayS14}, which in some
cases show that previous candidates are easy. One candidate by Raz
\cite{Raz00} that still stands is computing one entry of the
multiplication of $k$ $n\times n$ matrices over GF(2). He proves
\cite{BNS92}-like bounds for it, and further believes that this remains hard
even for $k$ much larger than $\log n$.  Our setting is different, for
example because we multiply more than $k$ matrices and the matrices
can be smaller.

We conjecture that over any non-abelian simple group, the interleaved
product remains hard even when the number of parties is more than
logarithmic in the input length. We note that this conjecture is interesting
even for a group of constant size and for deterministic protocols that
compute the whole product (as opposed to protocols that distinguish with
some advantage tuples that multiply to $g$ from those that multiply to
$h$).  We are unaware of upper bounds for this problem.

For context, we mention that the works \cite{BGKL03,PRS97,
Ambainis96,AmbainisL00} consider the so-called generalized addressing
function.  Here, the first $k-1$ parties receive an element $g_i$ from a
group $G$, and the last party receives a map $f$ from $G$ to $\zo$.  The
goal is to output $f(g_1 g_2 \cdots g_{k-1})$.  For any $k \ge 2$, this task
can be solved with communication $\log |G| + 1$.  Note that this is
logarithmic in the input length to the function which is $|G| + (k-1) \log |G|$.
By contrast, for interleaved products we prove and conjecture bounds that
are linear in the input length.  The generalized addressing function is more
interesting in restricted communication models, which is the focus of those
papers.

\paragraph{Application to leakage-resilient
cryptography.} We now informally describe the application
to cryptography we alluded to before -- for formal
definitions and extra discussion we refer the reader to
\cite{MilesV-leak}.  Also motivated by successful attacks
on cryptographic hardware, an exciting line of work known
as {\em leakage-resilient cryptography} considers models
in which the adversary obtains more information from
cryptographic algorithms than just their input/output
behavior. Starting with \cite{IshaiSW03}, a general goal
in this area is to compile any circuit into a new
``shielded'' circuit that is secure even if an adversary
can obtain partial information about the values carried
by the internal wires of the circuit during the
computation on inputs of their choosing.  This partial
information can be modeled in two ways.

One way is the ``only computation leaks'' model
\cite{MicaliR04}. Here the compiled circuit is
partitioned (by the compiler) into topologically ordered
sets of wires, i.e.\ in such a way that the value of each
wire depends only on wires in its set or in sets
preceding it. Goldwasser and Rothblum
\cite{GoldwasserR12} give a construction that is secure
against any leakage function that operates separately on
each set, as long as the function has bounded output
length.

Another way is the ``computationally bounded model,''
where the leakage function has access to all wires
simultaneously but it is computationally restricted
\cite{FaustRRTV10}.

The paper \cite{MilesV-leak} gives a new construction of
leakage-resilient circuits based on group products. This
construction enjoys strong security properties in the
``computationally bounded model''
\cite{MilesV-leak,Miles14}.
Moreover, the construction was shown to be secure even in
the ``only computation leaks'' model assuming that a
lower bound such as that in Theorem \ref{th-main-cc}
specialized to $k=8$ parties holds.

In this work we obtain such bounds and thus we also
obtain the following corollary.

\begin{corollary}
The leakage-resilient construction in \cite{MilesV-leak}
is secure in the ``only computation leaks'' model.
\end{corollary}
\begin{proof}
Combine Theorem \ref{th-main-cc} with Theorem 1.7 in
\cite{MilesV-leak}.
\end{proof}

This corollary completes the program of proving that the
construction in \cite{MilesV-leak} is secure in both the
``only computation leaks'' and the ``computationally
bounded'' models.  It seems to be the first construction
to achieve this.

\paragraph{Organization.}
This paper is organized as follows. In \S\ref{s-reductions} we exhibit a
series of reductions, valid in all groups, that reduce proving Theorem
\ref{t-t2-arbitrary} to Theorem \ref{t-sl2q-nice}. In \S\ref{s-boostin-p-i} we
use Theorem \ref{t-t2-arbitrary} to prove Theorem
\ref{t-intro-pairwise-uniform}. Then in \S\ref{s-nof-cc} we use Theorem
\ref{t-intro-pairwise-uniform} to prove the communication-complexity lower
bounds in Theorem \ref{th-main-cc}.  We also give some simple
equivalences between mixing and communication complexity, yielding the
mixing Theorem \ref{th-main-v1}.  Finally, Theorem \ref{t-sl2q-nice} is
proved in \S\ref{s-sl2q-nice}.  In \S\ref{s-misc} we give an alternative proof
of Theorem \ref{t-gowers-bnp} for \sl2q, and also prove Theorem
\ref{t-XYZ-cc}.

\section{Reducing Theorem \ref{t-t2-arbitrary} to Theorem \ref{t-sl2q-nice}} \label{s-reductions}

In this section we prove that Theorem \ref{t-t2-arbitrary} follows from
Theorem \ref{t-sl2q-nice}.  We need a definition and a couple of lemmas.
Here it is convenient to work with $L_2$-norms instead of $\ell_2$-norms,
and other norms we discuss will also be defined using averages rather than
sums.

\begin{definition}[Box norm]
Let $f : X_1 \times X_2 \times \cdots \times X_k \to \mathbb{R}$ be a
function. Define the \emph{box norm} $\|f\|_\square$ of $f$ as
\[\|f\|_\square^{2^k}=\E_{x_1^0,x_1^1,x_2^0,x_2^1,\dots,x_k^0,x_k^1}\prod_{\e\in\{0,1\}^k}
f(x_1^{\e_1},\dots,x_k^{\e_k}).\]
\end{definition}

In this section we only use the box norm with $k=2$ but later we will need
larger $k$ as well.

The first lemma is standard and says that a function with small box norm has
small correlation with functions of the form $(x,y)\mapsto u(x)v(y)$.

\begin{lemma} \label{boxnorminequality}
Let $X$ and $Y$ be finite sets, let $u:X\to\R$, let $v:Y\to\R$ and let
$f:X\times Y\to\R$. Then
\[|\E_{x,y}f(x,y)u(x)v(y)|\leq\|f\|_\square\|u\|_{L_2}\|v\|_{L_2}.\]
\end{lemma}

\begin{proof}
The proof uses two applications of the Cauchy-Schwarz inequality. We have
\[(\E_{x,y}f(x,y)u(x)v(y))^4&=((\E_xu(x)\E_yf(x,y)v(y))^2)^2\\
&\leq((\E_xu(x)^2)(\E_x(\E_yf(x,y)v(y))^2)^2\\
&=\|u\|_{L_2}^4\,(\E_{y,y'}v(y)v(y')\E_xf(x,y)f(x,y'))^2\\
&\leq\|u\|_{L_2}^4(\E_{y'y'}v(y)^2v(y')^2)\left(\E_{y,y'}(\E_xf(x,y)f(x,y'))^2\right)\\
&=\|u\|_{L_2}^4\|v\|_{L_2}^4\|f\|_\square^4.\]
The result follows on taking fourth roots.
\end{proof}

If we do not have a bound for $\|f\|_\square$, the next lemma shows that
we can still bound the correlation by bounding $\|ff^T\|_\square$.

\begin{lemma} \label{boxnormofsquare}
Let $X$ and $Y$ be finite sets, let $u:X\to\R$, let
$v:Y\to\R$ and let $f:X\times Y\to\R$. Let $g:X\times
X\to\R$ be defined by $g(x,x')=\E_yf(x,y)f(x',y)$. Then
\[|\E_{x,y}f(x,y)u(x)v(y)|\leq\|g\|_\square^{1/2}\|u\|_{L_2}\|v\|_{L_2}.\]
\end{lemma}
\begin{proof}
We have
\[(\E_{x,y}f(x,y)u(x)v(y))^2&=(\E_yv(y)\E_xf(x,y)u(x))^2\\
&\leq(\E_yv(y)^2) \E_y \left(\E_xf(x,y)u(x)\right)^2\\
&=\|v\|_{L_2}^2\,\E_{x,x'}(\E_yf(x,y)f(x',y))u(x)u(x').\\
&=\|v\|_{L_2}^2\,\E_{x,x'}g(x,x')u(x)u(x').\]
But by Lemma
\ref{boxnorminequality} this is at most
$\|v\|_{L_2}^2\|g\|_\square\|u\|_{L_2}^2$, which proves the lemma.
\end{proof}

We also use the following lemma about how the box norm is affected by adding a
constant function.

\begin{lemma} \label{squarebalanced}
Let $X$ and $Y$ be finite sets and let $F:X\times
Y\to\R$. Suppose that $\E_yF(x,y)=\d$ for every $x$ and
$\E_xF(x,y)=\d$ for every $y$. For each $x\in X$ and
$y\in Y$ let $f(x,y)=F(x,y)-\d$. Then
$\|f\|_\square^4=\|F\|_\square^4-\d^4$.
\end{lemma}
\begin{proof}
We have $F(x,y)=f(x,y)+\d$ for every $x$ and $y$. If we make this substitution
into the expression
\[\E_{x,x',y,y'}F(x,y)F(x,y')F(x',y)F(x',y'),\]
then we obtain 16 terms, of which two are
\[\E_{x,x',y,y'}f(x,y)f(x,y')f(x',y)f(x',y')\]
and $\d^4$. All remaining terms involve at least one variable that occurs
exactly once. Since $\E_y f(x,y)=0$ for every $x$ and $\E_x f(x,y)=0$ for every
$y$, all such terms are zero. The result follows.
\end{proof}

Now we are ready for the proof.

\begin{proof}[Proof of Theorem \ref{t-t2-arbitrary} assuming Theorem \ref{t-sl2q-nice}.]
It suffices to prove the theorem in the case where $g$ is
the identity element.  Let us pick $a$ and $b$ uniformly,
and note that what we want to bound equals
\[ |G|^4 | E_{a,b} (\Gamma(a,b) - 1/|G|) u(a)v(b)|,\]
where $\G(a,b)$ is the indicator function of $a \inpr b = 1$.  Letting \[f(a,b)
:= \Gamma(a,b) - 1/|G|,\] and \[g(x,x') := \E_y f(x,y) f(x',y),\] Lemma
\ref{boxnormofsquare} gives an upper bound of
\[
|G|^4 \| g \|^{1/2}_\square \|u\|_{L_2} \|v\|_{L_2}  = |G|^2 \| g \|^{1/2}_\square
\|u\|_{2} \|v\|_{2}.
\]
Now let us define
\[ \Delta(x,x') := \E_y \Gamma(x,y) \Gamma(x',y).\]

Note that for each $x$,
\[\E_{x'}\D(x,x')=\E_{x',y}\G(x,y)\G(x',y)=\E_y\G(x,y)\E_{x'}\G(x',y)=1/|G|^2.\]
By symmetry, $\E_x\D(x,x')=1/|G|^2$ for every $x'$ as well. Moreover, $g$
differs from $\Delta$ by a constant: $g(x,x') = \Delta(x,x') - 1/|G|^2$ for every $x,x'$
because
\[g(x,x') = \E_yf(x,y)f(x',y)=\E_y(\G(x,y)-1/|G|)(\G(x',y)-1/|G|) \\ =\E_y\G(x,y)\G(x',y)-1/|G|^2 = \Delta(x,x') - 1/|G|^2.\]

Hence we can apply Lemma \ref{squarebalanced} with $F$ replaced by
$\D$, $\d$ replaced by $1/|G|^2$, and $f$ replaced by $g$ to obtain
\[ \| g \|_\square^{1/2} \le (\| \Delta \|^4_\square - 1/|G|^8)^{1/8}. \]

Thus it remains to show
\[ \| \Delta \|^4_\square \le (1+\gamma)/|G|^8.\]
Note that
\[ \| \Delta \|^4_\square = \E_{x,x'} (\E_z \Delta(x,z)
\Delta(x',z))^2 = \E_{x,x'} (\E_z \Delta(x,z)
\Delta(z,x'))^2,\] where the first equality follows by
the definition of the box norm, and the second by the
fact that $\Delta$ is symmetric.

Now we fix $x$ and $x'$ and consider $\E_z \Delta(x,z)
\Delta(z,x') = \E_{z,y,y'} \Gamma(x,y) \Gamma(z,y)
\Gamma(z,y') \Gamma(x',y')$.  This is the probability,
for a randomly chosen $z,y,y'$ that
\[x_1y_1x_2y_2=z_1y_1z_2y_2=z_1y_1'z_2y_2'=x_1'y_1'x_2'y_2'=e,\]
which is $|G|^{-2}$ times the probability that
$x_1y_1x_2=z_1y_1z_2$ and $z_1y_1'z_2=x_1'y_1'x_2'$.

These last two equations can be rewritten as
\begin{align*}
y_1^{-1}z_1^{-1}x_1y_1 x_2 & = z_2 \\
y_1'^{-1}x_1'^{-1}z_1y_1'z_2 & = x_2'.
\end{align*}

By plugging the first equation in the second, and
right-multiplying by $x_2^{-1}$, we obtain that our
probability is $1/|G|$ times the probability that
\[ y_1'^{-1}x_1'^{-1}z_1y_1' y_1^{-1}z_1^{-1}x_1y_1 = x_2' x_2^{-1}.\]

We rewrite this as
\[ \uc(x_1'^{-1}z_1) \uc(z_1^{-1}x_1) = x_2' x_2^{-1}.\]

So we have shown that for every $x$ and $x'$:
\[
\E_z \Delta(x,z) \Delta(x',z) = \frac1{|G|^3} \Pr_z[\uc(x_1'^{-1}z_1) \uc(z_1^{-1}x_1) = x_2' x_2^{-1}].
\]
And consequently \[ \| \Delta \|^4_\square = |G|^{-6} \E_{x,x'} (\Pr_{z}
[\uc(x_1'^{-1}z_1) \uc(z_1^{-1}x_1) = x_2' x_2^{-1}])^2.
\]
Thus it remains to show that the expectation in the right-hand side is at
most $(1+\gamma)/|G|^2$.

By introducing variables $c = x_2' x_2^{-1}$, $b=z_1^{-1} x_1$, and $a =
x'^{-1}_1 x_1$ we can rewrite the expectation as
\[ E_{a,c} (\Pr_b [\uc(ab^{-1})\uc(b) = c])^2. \]

Multiplying by $|G|$, it remains to show that
\[ \sum_c E_a (\Pr_b [\uc(ab^{-1})\uc(b) = c])^2 \le
(1+\gamma)/|G|.\]

This follows from Theorem \ref{t-sl2q-nice}, which concludes the proof.
\end{proof}

\section{Boosting pairwise independence}
\label{s-boostin-p-i}

In this section we prove Theorem \ref{t-intro-pairwise-uniform} from the
introduction.  Actually, we prove a slightly stronger statement, Corollary
\ref{corollary-pairwise-to-uniform} below, which will be used later. First we fix
some notation.  Throughout the section $G$ is the group SL$(2,q)$ and
$|G|=n$.  For a real-valued function $f$, its $\ell_\infty$, and $\ell_1$
norms are respectively $\| f \|_\infty = \max_x |f(x)|$, and $\| f \|_1 = \sum_x
|f(x)|$. Next we define the measure of closeness to
uniform that we will work with.

\begin{definition}
A distribution $D$ on $G^m$ is $(\e,k)$-\emph{good} if for any $1\leq
i_1<\dots<i_k\leq m$ and any $g_1,\dots,g_k\in G$, the probability, when $x$ is
sampled randomly from $D$, that $x_{i_j}=g_j$ for $j=1,\dots,k$ is between
$(1-\e)n^{-k}$ and $(1+\e)n^{-k}$.
\end{definition}

To relate this definition to that of statistical distance, note that if a
distribution $D$ on $G^m$ is $(\e,k)$-good then the projection of $D$ to any $k$
coordinates is $\e$-close to uniform in statistical distance.

The main technical result shows how to go from pairwise independence to
three-wise independence.  We write $*$ for convolution, and note that the
convolution of two distributions $\mu$ and $\nu$ is the same as the
distribution obtained by sampling independently from $\mu$ and $\nu$ and
outputting the product.

\begin{theorem}
\label{main}  There is an integer $d \ge 2$ such that the following holds.
Let $\mu_1,\dots,\mu_d$ be $(1/\sqrt{n},2)$-good probability distributions
on $G^3$. Then $\mu_1*\dots*\mu_d$ is $(1/n^2,3)$-good.


\end{theorem}

The choice of polynomials $1/\sqrt{n}$ and $1/n^2$ will
be convenient in a later proof by induction, but is not
too important. Indeed, any bound of this type can be
quickly improved if one makes the products slightly
longer, as shown by the following lemma which we will use
several times.

\begin{lemma} \label{lemma-eps-good-to-eps-square}
Let $\mu$ and $\nu$ be $(\e,k)$-good probability
distributions on $G^m$. Then $\mu*\nu$ is
$(\e^2,k)$-good.
\end{lemma}
\begin{proof} The convolution of the projection of the distributions on to any $k$ coordinates is the same as the projection of the convolution, so it is enough to consider the case $m=k$. Let $H$ be any finite group (we shall be interested in the case $H=G^k$), let $U$ be the uniform distribution on $H$, and let $\mu$ and $\nu$ be distributions on $H$ such that $\|\mu-U\|_\infty$ and $\|\nu-U\|_\infty$ are both at most $\e/n$. Let $\a=\mu-U$ and $\b=\nu-U$. Then for every $x$ we have
\[\mu*\nu(x)=\sum_{yz=x}\left(\frac 1n+\a(y)\right)\left(\frac 1n+\b(z)\right)=\frac 1n+\sum_{yz=x}\a(y)\b(z).\]
where the second inequality follows from the fact that
$\a$ and $\b$ are functions that sum to zero.

But $|\sum_{yz=x}\a(y)\b(z)|\leq n(\e/n)^2=\e^2/n$, from
which it follows that $\|\mu*\nu-U\|_\infty\leq\e^2/n$.
Applying this when $H=G^k$ we obtain the result.
\end{proof}

Using the above two results and induction we obtain the
following simple corollary of Theorem \ref{main}.

\begin{corollary} \label{corollary-pairwise-to-uniform}
There is an integer $d \ge 2$ such that the following holds. Let $m\geq 3$,
and let $\mu_1,\dots,\mu_{d^m}$ be $(1/n,2)$-good probability
distributions on $G^m$, where $|G| = n$. Then $\mu_1*\dots*\mu_{d^m}$
is $(1/n,m)$-good.



\end{corollary}
\begin{proof}[Proof of Corollary \ref{corollary-pairwise-to-uniform} from Theorem \ref{main}]
In the proof we use that $(\e,k)$-good implies
$(\e,k-1)$-good, as can be seen by summing on one
coordinate.  We prove the corollary by induction on $m$.  For
$m=3$ this is Theorem \ref{main}.  Now we assume the
corollary for $m-1$ and prove it for $m$. Let $\nu_i$,
$i=1,\ldots,d$, be the product of $d^{m-1}$ consecutive
$\mu_i$. For an element $y$ of $G^m$ we write $y^0$ for
the first $m-3$ coordinates, and $y^1$ for the other
three. Pick any $x \in G^m$. We need to bound $\Pr[\nu_1*\dots*\nu_d = x]$,
which equals
\[ \label{eq-cor}
\Pr[\nu^1_1*\dots*\nu^1_d = x^1 | \nu^0_1*\dots\nu^0_d = x^0 ] \cdot
\Pr[ \nu^0_1*\dots*\nu^0_d = x^0 ].
\]

By induction, each $\nu^0_i$ is $(1/n,m-3)$-good, so
Lemma \ref{lemma-eps-good-to-eps-square} gives us
that $\nu^0_1*\dots*\nu^0_d$
is $(1/n^2,m-3)$-good.  Thus, the second term in expression \eqref{eq-cor}
lies between $(1-1/n^2)/n^{m-3}$
and $(1+1/n^2)/n^{m-3}$.

Now we bound the conditional probability in Equation
\eqref{eq-cor}.  Let $\a_i$ be the distribution $v^1_i$
on $G^3$ conditioned on any fixing of $v_i^0$.  We claim
that $\a_i$ is $(1/\sqrt{n},2)$-good.  Indeed, by
assumption the probability $p$ that two coordinates of
$\alpha_i$ equal any fixed pair satisfies
\[\frac{(1-1/n)/n^{m-1}}{(1+1/n)/n^{m-3}} \le p \le
\frac{(1+1/n)/n^{m-1}}{(1-1/n)/n^{m-3}}\] which implies
\[\frac{1-1/\sqrt{n}}{n^2} \le p \le
\frac{1+1/\sqrt{n}}{n^2}\] for large enough $n$.

Hence, by Theorem \ref{main} the convolution of the
$\a_i$ is $(1/n^2,3)$-good.

Putting together these bounds for the two factors in the
expression \eqref{eq-cor} we get that
the product lies between
$(1-1/n^2)^2/n^m$ and $(1+1/n^2)^2/n^m$, from which the
result follows.
\end{proof}

We remark that this corollary implies Theorem \ref{t-intro-pairwise-uniform}
stated in the introduction.

The rest of this section is devoted to the proof of Theorem \ref{main}.  To
prove the theorem we show that if we convolve pairwise uniform
distributions over $G^3$, then we reduce their $\ell_\infty$ norm.  To get a
sense of the parameters, note that the assumption of pairwise uniformity
implies an upper bound of $1/n^2$ on this norm, and that the minimum
possible value is $1/n^3$. So we are aiming to use convolutions to get
down from $1/n^2$ to about $1/n^3$.  Actually, it is more convenient to
work with the $\ell_2$ norm, but by convolving again we can return to the
$\ell_\infty$ norm thanks to the following simple fact.

\begin{fact} \label{fact-linftybound-from-l2}
For any distributions $\mu$ and $\nu$ it holds that $\|\mu*\nu\|_\infty \leq
\|\mu\|_{2}\|\nu\|_{2}$.
\end{fact}
\begin{proof}  For any $x$, $\mu * \nu (x) = \sum_y
\mu(y) \nu(y^{-1} x) \le \sqrt{\sum_y \mu(y)^2}
\sqrt{\sum_y \nu(y)^2}$, using the Cauchy-Schwarz
inequality.
\end{proof}

We now state and prove the flattening lemma.

\begin{lemma} \label{flattening}
Let $\mu$ and $\nu$ be two non-negative functions defined
on $G^3$ and suppose that however you fix two coordinates
of one of the functions and sum over the third, the total
is at most $n^{-2}$.
Then $\|\mu*\nu\|_{2}^2 \leq n^{-3}+n^{-\Omega(1)} \sqrt{\| \mu \|_\infty \|
\nu \|_\infty}$.
\end{lemma}

\begin{proof} Expanding out the definition of $\|\mu*\nu\|_{2}^2$ we obtain
\[\sum_{x_1y_1=z_1w_1}\sum_{x_2y_2=z_2w_2}\sum_{x_3y_3=z_3w_3}\mu(x_1,x_2,x_3)\nu(y_1,y_2,y_3)\mu(z_1,z_2,z_3)\nu(w_1,w_2,w_3).\]
We can rewrite this as
\[\sum_{a,b,x_1,x_2,y_1,y_2}\sum_{x_3y_3=z_3w_3}\mu(x_1a,x_2,x_3)\nu(y_1,by_2,y_3)\mu(x_1,x_2b,z_3)\nu(ay_1,y_2,w_3).\]
By averaging, it follows that there exist
$x_1,x_2,y_1,y_2$ such that
\[\|\mu*\nu\|_{2}^2\leq n^{4}\sum_{a,b}\sum_{x_3y_3=z_3w_3}\mu(x_1a,x_2,x_3)\nu(y_1,by_2,y_3)\mu(x_1,x_2b,z_3)\nu(ay_1,y_2,w_3). \label{eq-flattening-averaging}\]
Define $\a(a,x)$ to be $\mu(x_1a,x_2,x)$, $\b(b,y)$ to be $\nu(y_1,by_2,y)$,
$\g(b,z)$ to be $\mu(x_1,x_2b,z)$ and $\d(a,w)$ to be $\nu(ay_1,y_2,w)$. Then we
can rewrite this inequality as
\[\|\mu*\nu\|_{2}^2\leq n^{4}\sum_{a,b}\sum_{xy=zw}\a(a,x)\b(b,y)\g(b,z)\d(a,w).\]
Now let us set $u(x,w)$ to be $\sum_a\a(a,x)\d(a,w)$ and
$v(y,z)$ to be $\sum_b\b(b,y)\g(b,z)$.

Our bound \eqref{eq-flattening-averaging} can be rewritten as
\[\|\mu*\nu\|_{2}^2\leq n^{4}\sum_{xy=zw}u(x,w)v(y,z).\]
On the right-hand side there is an interleaved product of the kind to which
Theorem \ref{t-t2-arbitrary} can be applied.  To apply it, we proceed to bound
the norms of $u$ and $v$.

Note that by our hypotheses on $\mu$ and $\nu$ we have for each $w$ that
\[\sum_xu(x,w)=\sum_a\d(a,w)\sum_x\a(a,x)=\sum_a\d(a,w)\sum_x\mu(x_1a,x_2,x)\leq n^{-2}\sum_a\d(a,w)\leq n^{-4},\]
with three similar inequalities for summing over the other coordinate and for
$v$.  Hence we also have that $\sum_{x,w}u(x,w)\leq n^{-3}$. We also have for
each $x,w$ that
\[u(x,w)\leq\|\a\|_\infty\sum_a\d(a,w) \le \| \mu \|_\infty n^{-2},\]
with a similar argument giving the same bound for $\|v\|_\infty$.  Combining
these two facts we get that $\sum_{x,w}u(x,w)^2\leq \| \mu \|_\infty / n^5$. And
we have a similar bound for $v$.

We apply Theorem \ref{t-t2-arbitrary} to the probability distributions
$u/\|u\|_1$ and $v/\|v\|_1$, and then we multiply by $\|u\|_1 \|v\|_1$ to obtain
that
\[\sum_{xy=zw}u(x,w)v(y,z)=n^{-1}\|u\|_1\|v\|_1+ n \cdot \gamma \|u\|_{2}\|v\|_{2} \leq n^{-1} \cdot n^{-3} \cdot n^{-3}  + \gamma \sqrt{\|\mu \|_\infty \| \nu \|_\infty}/n^4.\]
This implies that
\[\|\mu*\nu\|_{2}^2 \leq n^{-3}+\gamma \sqrt{\|\mu \|_\infty \| \nu \|_\infty},\]
which proves the result.
\end{proof}

\begin{corollary} \label{ellinfinityversion}
Let $\mu_1,\mu_2,\mu_3,\mu_4$ be non-negative functions
defined on $G^3$ and suppose that they all satisfy the
condition that $\mu$ and $\nu$ satisfy in Lemma
\ref{flattening}.  Suppose further that $\| \mu_i
\|_\infty \le \alpha$ for every $i \in \{1,2,3,4\}$.

Then
\[\|\mu_1*\mu_2*\mu_3*\mu_4\|_\infty \leq n^{-3}+n^{-\Omega(1)} \alpha.\]
\end{corollary}

\begin{proof}
This follows on applying Lemma \ref{flattening} to
$\mu_1$ and $\mu_2$ and to $\mu_3$ and $\mu_4$ and then
applying Fact \ref{fact-linftybound-from-l2} to
$\mu_1*\mu_2$ and $\mu_3*\mu_4$.
\end{proof}

The next lemma shows that convolution preserves one of
the main properties we used.

\begin{lemma} \label{rowsums}
Let $\mu$ and $\nu$ be non-negative functions defined on
$G^3$ and suppose that whenever you fix two coordinates
of $\mu$ or $\nu$ and sum over the other, you get at most
$n^{-2}$. Then the same is true of $\mu*\nu$.
\end{lemma}

\begin{proof}
For each $(z_1,z_2,z_3)\in G$ we have
\[\mu*\nu(z_1,z_2,z_3)=\sum_{x_1}\sum_{x_2}\sum_{x_3}\mu(x_1,x_2,x_3)\nu(x_1^{-1}z_1,x_2^{-1}z_2,x_3^{-1}z_3).\]
If we fix $z_1$ and $z_2$ and sum over $z_3$ we obtain
\[\sum_{x_1}\sum_{x_2}\sum_{x_3,z_3}\mu(x_1,x_2,x_3)\nu(x_1^{-1}z_1,x_2^{-1}z_2,x_3^{-1}z_3).\]
But for each $x_1,x_2$, we have
\[\sum_{x_3,z_3}\mu(x_1,x_2,x_3)\nu(x_1^{-1}z_1,x_2^{-1}z_2,x_3^{-1}z_3)=\sum_x\mu(x_1,x_2,x)\sum_y\nu(x_1^{-1}z_1,x_2^{-1}z_2,y)\leq n^{-4}.\]
The result follows on summing over $x_1$ and $x_2$.
\end{proof}

\begin{proof}[Proof of Theorem \ref{main}.]
If we divide each $\mu_i$ by $(1+1/\sqrt{n})$, then we
obtain functions $\nu_i$ that satisfy the conditions of
Lemma \ref{flattening} and such that $\| \nu_i \|_\infty$
is at most $1/n^2$.  Applying Corollary
\ref{ellinfinityversion} a constant number of times,
using Lemma \ref{rowsums} to argue that the assumptions
are satisfied throughout, we deduce that a convolution of
a constant number $\ell$ of such functions has infinity
norm at most $n^{-3}(1+n^{-\Omega(1)})$. If we now
multiply one such convolution by $(1+1/\sqrt{n})^\ell$ we
obtain a probability distribution $\mu$ with
\[\|\mu\|_\infty \leq
n^{-3}(1+n^{-\Omega(1)})(1+1/\sqrt{n})^\ell
\le n^{-3}(1 + n^{-\Omega(1)})\] for large enough $n$.

This is close to our goal of bounding $\|\mu - U\|_\infty$.  To achieve the
goal, we use the following fact about any two probability distributions $\a$
and $\b$ over $G^3$, where note the inequality is Fact
\ref{fact-linftybound-from-l2}:
\[ \| \a*\b - U \|^2_\infty = \| (\a-U)*(\b-U)\|^2_\infty \\ \le
\|\a-U\|_{2}^2 \|\b-U\|_{2}^2 = (\| \a \|_{2}^2 - 1/n^3)(\| \b
\|_{2}^2 - 1/n^3).\]

In our case we have $\| \mu \|_{2}^2 \le \|\mu\|_\infty \leq n^{-3}(1 +
n^{-\Omega(1)})$.  So we convolve one more time and apply the above
fact to obtain a distribution $\mu'$ such that $\| \mu' - U \|_\infty \le
n^{-\Omega(1)}/n^3$.  Hence, $\mu'$ is $(n^{-\Omega(1)},3)$-good.  Now if
we convolve another constant number of times and apply Lemma
\ref{lemma-eps-good-to-eps-square} we obtain a distribution which is
$(n^{-2},3)$-good, as desired.
\end{proof}

\section{Communication complexity lower bound} \label{s-nof-cc}

In this section we prove Theorem \ref{th-main-cc}.  We shall focus first on the
case where $t$ is at least a large enough constant.  The only case that is not
covered by this is the case of $t=k=2$.  We then establish some simple
equivalences that give that case, and also Theorem \ref{th-main-v1}.

A key idea is to obtain this theorem by showing that a certain collection of
group products are jointly close to uniform.  The group products to consider
arise naturally from an application of the ``box norm'' (a.k.a.~the multiparty
norm).  The next theorem summarizes this result.

\begin{theorem} \label{t-tuples-equidistr}
There is a constant $b$ such that the following holds.  Let $k$ and $t$ be
integers, $G = \sl2q$, $m = 2^k$, and $t \ge b^m$.
Let $x_1^0,x_1^1,\dots,x_k^0,x_k^1$ be chosen independently and uniformly from
$G^t$ and consider the distribution $\mu$ on $G^m$ whose coordinate
$\e\in\{0,1\}^k$ is the interleaved product
\[x_1^{\e_1}\bullet x_2^{\e_2}\bullet\dots\bullet x_k^{\e_k}.\]
Then $\mu$ is $(1/|G|^{t/b^m},m)$-good.
\end{theorem}

To apply Corollary \ref{corollary-pairwise-to-uniform} we need to show that our
distributions can be written as the product of many pairwise-independent
distributions.  This is done by the following lemma.

\begin{lemma} \label{lemma-box-norm-is-product-pairwise}
Let $\mu$ be the distribution over $G^m$ in Theorem \ref{t-tuples-equidistr},
and let also $t$ be as in Theorem \ref{t-tuples-equidistr}.  Then $\mu$ is the
component-wise product of $t$ independent distributions, each of which is
pairwise uniform.
\end{lemma}
\begin{proof}
Recall that $m = 2^k$.  Let us write
$x_i^{\e_i}=(a_{i1}^{\e_i},\dots,a_{ik}^{\e_i})$. Then for each $\e\in\{0,1\}^k$
we have
\[x_1^{\e_1}\bullet x_2^{\e_2}\bullet\dots\bullet x_k^{\e_k}=a_{11}^{\e_1}\dots a_{k1}^{\e_k}a_{12}^{\e_1}\dots a_{k2}^{\e_k}\dots a_{1t}^{\e_1}\dots a_{kt}^{\e_k}.\]
Now for $1\leq j\leq t$ let $s_j$ be the $m$-tuple $(a_{1j}^{\e_1}\dots
a_{kj}^{\e_k})_{\e\in\{0,1\}^k}$. Then the $m$-tuple $(x_1^{\e_1}\bullet
x_2^{\e_2}\bullet\dots\bullet x_k^{\e_k})_{\e\in\{0,1\}^k}$ is the pointwise
product of the $s_j$. That is, writing $s_j(\e)$ for $a_{1j}^{\e_1}\dots
a_{kj}^{\e_k}$, we have that
\[x_1^{\e_1}\bullet x_2^{\e_2}\bullet\dots\bullet x_k^{\e_k}=s_1(\e)s_2(\e)\dots s_t(\e)\]
for every $\e\in\{0,1\}^k$.

Note that the $m$-tuples $s_j$ are independent and distributed as follows. We
choose elements $u_1^0,u_1^1,u_2^0,u_2^1,\dots,u_k^0,u_k^1$ uniformly and
independently at random from $G$ and we form an $m$-tuple $s$ by setting
$s(\e)=u_1^{\e_1}u_2^{\e_2}\dots u_k^{\e_k}$.

We note that $s$ is pairwise uniform. That is, if you take any pair of distinct
elements $\e, \eta$ in $\{0,1\}^k$, then the pair $(s(\e),s(\eta))$ is uniformly
distributed in $G^2$. To see this, choose some $i$ such that $\e_i\ne\eta_i$.
Conditioning on the values of $u_j^{\e_j}$ and $u_j^{\eta_j}$ for every $j\ne
i$, we find that we are looking at two products of the form $au_i^{\e_i}b$ and
$cu_i^{\eta_i}d$. For this to equal $(g,h)$, we need $u_i^{\e_i}=a^{-1}gb^{-1}$
and $u_i^{\eta_i}=c^{-1}hd^{-1}$. Since $u_i^{\e_i}$ and $u_i^{\eta_i}$ are
independent and uniformly distributed, this happens with probability $1/|G|^2$.
\end{proof}

We note that the distribution $s$ in the proof is far from being uniformly
distributed, since there are only $n^{2k}$ possible $2^k$-tuples of this form.

Now Theorem \ref{t-tuples-equidistr} follows easily.

\begin{proof}[Proof of Theorem \ref{t-tuples-equidistr}]
Let $m$ be $2^k$.  Let $d$ be the constant in Corollary
\ref{corollary-pairwise-to-uniform}.  Write the distribution $\mu$ as the
product of $t$ independent distributions $\mu_i$, each of which is pairwise
uniform, using Lemma \ref{lemma-box-norm-is-product-pairwise}. Group the $\mu_i$
in consecutive blocks of length $d^m$. The convolution in each block is
$(1/n,m)$-good by Corollary \ref{corollary-pairwise-to-uniform}. By repeated
applications of Lemma \ref{lemma-eps-good-to-eps-square} we obtain that the
final distribution is $(1/n^{t/b^m},m)$-good.  The change of the constant from
$d$ to $b$ is to handle the case in which $t/d^m$ is not a power of two.
\end{proof}

Finally, the proof that Theorem \ref{t-tuples-equidistr} implies Theorem
\ref{th-main-cc} is a technically simple application of the ``box norm,'' given
next.

\begin{proof}[Proof that Theorem \ref{t-tuples-equidistr} implies Theorem
\ref{th-main-cc}]


Consider the function $d : G^{tk}\to\{0,1,-1\}$ that maps $x=(x_1,\dots,x_k)$ to
1 if $x_1\bullet\dots\bullet x_k=e$, to $-1$ if $x_1\bullet\dots\bullet x_k=g$,
and to 0 otherwise.  Then we have
\[|p_g - p_h| = 0.5 \cdot n \cdot |\E_x (-1)^{P(x)} d(x)|.\]

Following previous work \cite{BNS92,ChT93,Raz00,ViW-GF2}, we bound the latter
expectation using the \emph{box norm} $\|d\|_\square$ of $d$.

Specifically, by e.g.~Corollary 3.11 in \cite{ViW-GF2}, we have
\[ 0.5 \cdot n \cdot |\E_x (-1)^{P(x)} d(x) | \le 0.5 \cdot n \cdot 2^c \cdot \|d\|_\square.\]

To conclude it remains to notice that Theorem \ref{t-tuples-equidistr} allows us
to bound $\|d\|_\square$.  First, note that the product in $\|d\|_\square^{2^k}$
is equal to zero unless each of the $2^k$ interleaved products
$x_1^{\e_1}\bullet\dots\bullet x_k^{\e_k}$ is equal to 1 or $g$, in which case
it is $1$ if the number of products equal to $g$ is even and $-1$ if it is odd.
If the $2^k$ products were uniform and independent, then
the expectation of the product would be zero. If instead they are
$(\alpha,2^k)$-good
then $\|d\|_\square^{2^k}\le 2^k \alpha/n^{2^k}$, and so $\|d\|_\square\le 2
\alpha^{1/2^k}/n$.  Plugging in $\alpha = 1/|G|^{t/b^m}$ for $m=2^k$ completes
the proof.

\end{proof}

\subsection{The remaining claims}

We now establish some simple equivalences that give Theorem
\ref{th-main-cc}.(2), and also Theorem \ref{th-main-v1}. Specifically we
shall show that Theorem \ref{th-main-cc}.(2) and the mixing bound for flat
distributions, Theorem \ref{th-main-v1}, are both equivalent to the following
version of the mixing bound.  We identify sets with their characteristic
functions.

\begin{theorem} \label{th-main-v2}  Let $G = \sl2q$.  Let $A, B \subseteq G^t$ have densities $\alpha$ and
$\beta$ respectively.  Let $g \in G$.  We have
\[ | \E_{a \inpr b = g} A(a) B(b) - \alpha
\beta | \le |G|^{-\Omega(t)},\] where the expectation is over $a$ and $b$ such that $a \inpr b = g$.
\end{theorem}

\begin{claim} \label{claim-equivalences}  Theorems \ref{th-main-cc},
\ref{th-main-v1}, and \ref{th-main-v2} are equivalent.
\end{claim}

Given this claim we obtain Theorem \ref{th-main-cc}.(2) from Theorem
\ref{t-t2-arbitrary}, and we obtain Theorem \ref{th-main-v1} from Theorem
\ref{th-main-cc}.(1).

\begin{proof}[Proof of Claim \ref{claim-equivalences}]
The equivalence between the two versions of the mixing bound, Theorems
\ref{th-main-v1} and \ref{th-main-v2}, follows by Bayes' identity
\[ \Pr[a \inpr b = g|a \in A, b \in B] = \frac{\Pr[a \in A, b \in B | a \inpr b = g]}{|G| \alpha \beta}.\]

We now show that Theorem \ref{th-main-v2} implies the communication bound,
Theorem \ref{th-main-cc}.  By an averaging argument we can assume that the
protocol $P$ in Theorem \ref{th-main-cc} is deterministic.  Now write
\[P( a,b ) = \sum_{i \le C} R_i( a,b )\] where $C =
2^c$, the $R_i$ are disjoint rectangles in $(G^t)^2$, i.e., $R_i = S_i \times
T_i$ for some $S_i, T_i \subseteq G^t$, cf.~\cite{KuN97}, and we also write
$R_i$ for the characteristic function with output in $\zo$.  For any $g$ and $h$
in $G$ we then have, using the triangle inequality:
\[ |p_g - p_h| = \left| \sum_{i \le C} \left( \E_{a \inpr b = g}
R_i(a,b) - |R_i|/|G|^{2t} + |R_i|/|G|^{2t} - \E_{a \inpr b = h} R_i(a,b) \right) \right| \le 2^C
|G|^{-\Omega(t)}.\]

To see the reverse direction, that Theorem \ref{th-main-cc} implies Theorem
\ref{th-main-v2}, suppose that we are given sets $A$ and $B$.  Consider the
constant-communication protocol $P(a,b) = A(a)B(b)$, and note that $p_g = \E_{a
\inpr b = g} A(b)B(b)$ and that $\E_h p_h = \alpha \beta$. So we have
\[ | \E_{a \inpr b = g} A(a) B(b) - \alpha
\beta | = |p_g - \E_h p_h| \le \E_h |p_g - p_h| = O(|G|^{-\Omega(t)}).\]
\end{proof}

\section{\sl2q: Proof of Theorem \ref{t-sl2q-nice}} \label{s-sl2q-nice}

In this section we prove Theorem \ref{t-sl2q-nice}.  We start with a lemma
from the literature.

\begin{lemma}[Structure of SL$(2,q)$.] \label{l-sl2q}
The group SL$(2,q)$ has the following properties.

1. It has size $q^3 - q$.

2. It has $q + O(1)$ conjugacy classes.

3.  All but $O(1)$ conjugacy classes have size either $q(q+1)$ or $q(q-1)$.

4.  Every conjugacy class has size $\Omega(q^2)$, except for the trivial classes
$\{1_G\}$ and $\{-1_G\}$ which have size 1.
\end{lemma}

\noindent Property 1 is easy to verify.  For more precise versions of Properties 2 and 3 see
e.g.~theorems 38.1 and 38.2 in \cite{Dornhoff71}.  Next we state another
lemma and then prove Theorem \ref{t-sl2q-nice} from the lemmas.

\begin{lemma} \label{lemma-cc-equidistribution} Let $G = \sl2q$.
Let $D = (D_1,D_2)$ be a distribution over $G^2$ such that $D_1$ is uniform and
$D_2$ is uniform.  With probability $1-O(1/q)$ over a pair $(g,h)$ sampled from
$D$ the following holds.

(1) For any conjugacy class $S$ of $G$, the probability that $g \uc(h) \in
S$ is $O(1/q)$.

(2) The distribution of $\uc(g)\uc(h)$ is $q^{-\Omega(1)}$-close to uniform
over $G$.
\end{lemma}

\begin{proof}[Proof of Theorem \ref{t-sl2q-nice} assuming Lemma \ref{lemma-cc-equidistribution}]
First, note that the bound claimed in the theorem is equivalent to
\[ \sum_c (\Pr_b [\uc(ab^{-1})\uc(b) = c] - 1/|G|)^2 \le \gamma^{\Omega(1)}/|G|.\]

We proceed by case analysis. The $c = 1$ summand is, by Cauchy-Schwarz,
\[(\Pr_b [ab^{-1}\uc(b) = 1] - 1/|G|)^2 \le \E_b (\Pr[ab^{-1} \uc(b) = 1]-1/|G|)^2.\]
If $b \not \in \{1,-1\}$ then the conjugacy class of $b$ has size
$\Omega(q^2)$, by Lemma \ref{l-sl2q}. Hence, $\Pr[ab^{-1} \uc(b) = 1] =
O(1/q^2)$ and so $(\Pr[ab^{-1} \uc(b) = 1]-1/|G|)^2 \le O(1/q^4)$.  If instead
$b \in \{1,-1\}$ then the conjugacy class of $b$ is $\{b\}$, so the probability
is $1$ if $a=1$, which happens with probability $1/|G|$, and $0$ otherwise.
Thus, the $c=1$ summand is at most $O(1/q^4) + O(1/|G|^2) \le
\gamma^{\Omega(1)}/|G|$.

A similar argument gives the same bound for the $c = -1$ summand.

It remains to show that
\[ \sum_{c \in G \setminus \{-1,1\}} (\Pr_b [\uc(ab^{-1})\uc(b) = c] - 1/|G|)^2 \le
\gamma^{\Omega(1)}/|G|.\] We bound above the left-hand side of this by
\[ \left(\max_{c \in G \setminus \{-1,1\}} (\Pr_b [\uc(ab^{-1})\uc(b) = c] - 1/|G|)  \right)
\left(\sum_{c \in G \setminus \{-1,1\}} (\Pr_b [\uc(ab^{-1})\uc(b) = c] - 1/|G|) \right).\]

First we show that the maximum is $O(1/|G|)$. Except with probability
$O(1/q)$ over the choice of $b$, Lemma \ref{lemma-cc-equidistribution}
guarantees that $\uc(ab^{-1})\uc(b)$ is in the conjugacy class of $c$ with
probability $O(1/q)$. Because by Lemma \ref{l-sl2q} this class has size
$\Omega(q^2)$, the probability that $\uc(ab^{-1})\uc(b) =
\uc(\uc(ab^{-1})\uc(b))$ equals $c$ is $O(1/q) O(1/q^2) = O(1/|G|)$.  In the
event that we cannot apply the lemma, the probability is still at most
$O(1/q^2)$.  Hence overall the maximum is at most $O(1/|G|)$.

Thus, it remains to show that
\[ \sum_{c \in G
\setminus \{-1,1\}} |\Pr_b [\uc(ab^{-1})\uc(b) = c] - 1/|G||
\le \gamma^{\Omega(1)}.\] The left-hand side of this is at most
\[ \E_b \sum_c |\Pr [\uc(ab^{-1})\uc(b) = c] - 1/|G||.\]
Except with probability $O(1/q)$ over $b$, we have by Lemma
\ref{lemma-cc-equidistribution} that the distribution $\uc(ab^{-1})\uc(b)$ is
$q^{-\Omega(1)}$-close to uniform over $G$, in which case the sum is at
most $q^{-\Omega(1)}$.  Also, for every $b$ the sum is at most $2$.  So
overall we obtain an upper bound of $q^{-\Omega(1)} + O(1/q)$, as
desired.
\end{proof}

To complete the proof of Theorem \ref{t-sl2q-nice} we must prove Lemma
\ref{lemma-cc-equidistribution}.  A key observation, which is also central to
many other papers concerning conjugacy classes in \sl2q, is that there is an
approximate one-to-one correspondence between conjugacy classes and the
{traces} of the matrices in the conjugacy class. In one direction this is
trivial, since the trace is a conjugacy invariant.  The other direction can be
expressed in a form suitable for our purposes as the following claim.

\begin{definition}
For a group element $g$ we denote by $\tr(g)$ the \emph{trace} of the matrix
corresponding to $g$, and by $\cc(g)$ the conjugacy class of $g$.
\end{definition}

\begin{claim} \label{c-TrUnidImpliesCCUnif} Let $G = \sl2q$, let $D$ be a distribution
over $G$ and let $U$ be the uniform distribution over $G$. Suppose that
$\Tr(D)$ is $\e$-close to uniform over $\F_q$ in statistical distance.  Then
$\cc(D)$ and $\cc(U)$ are $\e'$-close, where $\e' = O(\e) + O(1/q)$.
\end{claim}
\begin{proof}
Let $H$ be the set of conjugacy classes of $G$ which have size in
$\{q(q+1),q(q-1)\}$ and whose trace is not equal to that of any other conjugacy
class.  We have that $|H| \ge q-O(1)$, because all $q$ field elements can arise
as the trace of a conjugacy class, the number of conjugacy classes is $q+O(1)$
by Lemma \ref{l-sl2q}, and all but $O(1)$ conjugacy classes have size in
$\{q(q+1),q(q-1)\}$ again by Lemma \ref{l-sl2q}.

Next we claim that the distribution of $\cc(U)$ is $O(1/q)$-close to the uniform
distribution $V$ over $H$. Indeed, $\Pr[\cc(U) = c] = (q^2 + e_c)/|G|$ where
$|e_c| = q$ for any $c \in H$.  And for any $c \not \in H$ we have $\Pr[\cc(U) =
c] = O(q^2/|G|) = O(1/q)$. Hence the statistical distance between $\cc(U)$ and
$V$ is at most
\[ \sum_{c \not \in H} O(1/q) + \sum_{c \in H} \left|
\frac{q^2 + e_c}{|G|} - \frac1{|H|} \right|
\le O(1/q) + \sum_{c \in H} \left| \frac{q^2}{|G|} - \frac1{|H|} \right| \le O(1/q).\]

Finally, we claim that $\cc(D)$ is $\e'$-close to $V$. The probability that
$\cc(D) = c$ for a $c$ in $H$ is equal to the probability that $\Tr(D) = c$.
Let $B$ be the set of $q-|H| = O(1)$ values for the trace map that do not
arise from classes in $H$. The statistical distance between $\cc(U)$ and
$V$ is at most
\[ \Pr[\tr(D) \in B] + \sum_{c \in H} |\Pr[\tr(D) = c] -
1/|H|| \\ \le \e + \sum_{c \in H} |\Pr[\tr(D) = c] - 1/q| +  \sum_{c \in H} |1/q -
1/|H|| \le O(\e) + O(1/q).\]

The result follows by summing the distance between $\cc(U)$ and $V$ and
between $\cc(D)$ and $V$.
\end{proof}

This correspondence allows us to derive Lemma \ref{lemma-cc-equidistribution}
from the following lemma about the trace map.

\begin{lemma} \label{lemma-trace-equidistribution}
Let $G=SL(2,q)$.  Let $v$ and $w$ be two elements of $\F_q$.  Suppose
that either (i) $q$ is even, or (ii) $q$ is odd and $(v^{2},w^{2})\ne(-4,-4)$
and $(v,w)\ne(0,0)$. Let $D$ be the distribution of $Tr\left(
\left(\begin{array}{cc}
0 & 1\\
1 & w
\end{array}\right)\uc\left(\begin{array}{cc}
v & 1\\
1 & 0
\end{array}\right)\right)$.  Then

(1) $D$ takes any value $x \in \F_q$ with probability $O(1/q)$, and

(2) $D$ is $1/q^{\Omega(1)}$ close to uniform in statistical distance.

\end{lemma}

\begin{proof}[Proof of Lemma
\ref{lemma-cc-equidistribution}] Note that for every $h$ and $g$ the
distribution of the trace of $h u g u^{-1}$ for uniform $u$ is the same as the
distribution of the trace of $h' u g' u^{-1}$ for uniform $u$, for any $h'$ that
is conjugate to $h$ and for any $g'$ that is conjugate to $g$.  This is true
because if $g=xg'x^{-1}$ and $h=yh'y^{-1}$ then by the cyclic-shift property of
the trace function we have
\[ \tr(yh'y^{-1} u xg'x^{-1} u^{-1}) = \tr( h'y^{-1} u xg'x^{-1}
u^{-1} y),\] and the latter has the same distribution of the trace of $h' u g'
u^{-1}$ for uniform $u$.  Because of this fact, Lemma
\ref{lemma-trace-equidistribution} applies to $\tr (g \uc(h))$ for any $g$
except those in $O(1)$ conjugacy classes and similarly for any $h$ except
those in $O(1)$ conjugacy classes.  Those conjugacy classes make up at
most an $O(1/q)$ fraction of the group.  Hence, with probability $1-O(1/q)$
over $(g,h)$ sampled from $D$, we can apply Lemma
\ref{lemma-trace-equidistribution}.

Property (1) in Lemma \ref{lemma-trace-equidistribution} immediately gives Property (1)
in Lemma \ref{lemma-cc-equidistribution}.

To verify Property (2), note that the distribution of $\uc(g)\uc(h)$ is the same
as that of $\uc(\uc(g)\uc(h))$.  By Item (2) in Lemma
\ref{lemma-trace-equidistribution}, $\tr (\uc(g)\uc(h))$ is
$q^{-\Omega(1)}$-close to uniform. By Claim \ref{c-TrUnidImpliesCCUnif},
$\cc (\uc(g)\uc(h))$ is $q^{-\Omega(1)}$-close to the distribution of the
conjugacy class of a uniform element from $G$.  Hence $\uc(\uc(g)\uc(h))$
is $q^{-\Omega(1)}$-close to uniform.
\end{proof}

It remains to prove Lemma \ref{lemma-trace-equidistribution}.  This proof is
somewhat technical and appears in the next subsection.

\subsection{Proof of Lemma
\ref{lemma-trace-equidistribution}}

In this subsection, if we use a letter such as $a$ to refer to an element
of $G$, we shall refer to its entries as $a_1,\dots,a_4$. That is, we
shall take $a$ to be the matrix $\begin{pmatrix}a_1&a_2\\a_3&a_4\end{pmatrix}$.

We begin by working out an expression for the trace that concerns us.

\begin{claim} \label{claim-trace-expr}
Let $a,u$ and $g$ be $2\times 2$ matrices in SL$(2,q)$. Then
\[\tr(a u g u^{-1}) =
(a_1u_1+a_2u_3)(g_1u_4-g_2u_3)+(a_1u_2+a_2u_4)(g_3u_4-g_4u_3) \\
+(a_3u_1+a_4u_3)(-g_1u_2+g_2u_1)+(a_3u_2+a_4u_4)(-g_3u_2+g_4u_1).\]
\end{claim}
\begin{proof}
Note that $\begin{pmatrix}u_1&u_2\\u_3&u_4\end{pmatrix}^{-1} =
\begin{pmatrix}u_4&-u_2\\-u_3&u_1\end{pmatrix}$.
Now
\[a u = \begin{pmatrix}a_1&a_2\\a_3&a_4\end{pmatrix} \begin{pmatrix}u_1&u_2\\u_3&u_4\end{pmatrix} = \begin{pmatrix}a_1u_1+a_2u_3&a_1u_2+a_2u_4\\a_3 u_1 + a_4
u_3&a_3u_2+a_4u_4\end{pmatrix} \] and
\[g u^{-1} = \begin{pmatrix}g_1&g_2\\g_3&g_4\end{pmatrix} \begin{pmatrix}u_4&-u_2\\-u_3&u_1\end{pmatrix} = \begin{pmatrix}g_1u_4-g_2u_3&-g_1u_2+g_2u_1\\
g_3 u_4 - g_4 u_3&-g_3u_2+g_4u_1\end{pmatrix}. \] The result follows.
\end{proof}

Our proof of Lemma \ref{lemma-trace-equidistribution} uses the following
well-known theorem from arithmetic geometry, due to Lang and Weil
\cite{LangWeil54}. It can also be found as Theorem 5A, page 210, of
\cite{Schmidt2004equations}.

\begin{theorem} \label{le-lang-weil}
For every positive integer $d$ there is a constant $c_d$ such that the following
holds: if $f(x_1,\ldots,x_n)$ is any absolutely irreducible polynomial over
$F_q$ of total degree $d$, with $N$ zeros in $F^n_q$, then
\[|N - q^{n-1}| \leq c_dq^{n-3/2}.\]
\end{theorem}

After these preliminaries we now present the proof of Lemma
\ref{lemma-trace-equidistribution}. First we remark that the calculation below
for the trace in the case $v=w=0$ shows that the condition $(v,w)\ne(0,0)$ is
necessary over odd characteristic.

From Claim \ref{claim-trace-expr} we obtain the following expression for the
trace.
\begin{multline*}
f'':=u_{3}(vu_{4}-u_{3})+u_{4}u_{4}+(u_{1}+wu_{3})(-vu_{2}+u_{1})+(u_{2}+wu_{4})(-u_{2})\\
=vu_{3}u_{4}-u_{3}^{2}+u_{4}^{2}-vu_{1}u_{2}+u_{1}^{2}-vwu_{2}u_{3}+wu_{1}u_{3}-u_{2}^{2}-wu_{2}u_{4}.
\end{multline*}

We shall show that for all but $O(1)$ choices for $s$, the number of solutions
to the system $f''=-s$ and $u_{1}u_{4}-u_{2}u_{3}=1$ has distance $e_{s}$ from
$q^{2}$ where $|e_{s}|\le q^{2-\Omega(1)}$. And for the other $O(1)$ choices of
$s$ the number of solutions is $O(q^{2})$.  This will show Property (2), i.e., that
the trace has statistical distance $1/q^{\Omega(1)}$ from uniform. Indeed, using
that $|G|=q^{3}-q$, the contribution to this distance from each of the $q-O(1)$ good
values of $s$ is $|(q^{2}+e_{s})/(q^{3}-q)-1/q|=|(1+e_{s})/(q^{3}-q)|\le1/q^{1+\Omega(1)}$
because $|e_{s}|\le q^{2-\Omega(1)}$. These add up to a contribution of
$1/q^{\Omega(1)}$, while for each of the others the contribution is at most
$O(1/q)$.  Property (1) will then follow.

First, we consider the case when $q$ is even and $v=w=0.$ In this case the trace
becomes
\[
(u_{1}-u_{2}-u_{3}+u_{4})^{2}.
\]
Now note that the map $\left(\begin{array}{cc}
u_{1} & u_{2}\\
u_{3} & u_{4}
\end{array}\right)\rightarrow\left(\begin{array}{cc}
u_{1} & u_{2}\\
u_{3}+u_{1} & u_{4}+u_{2}
\end{array}\right)$ is a permutation on $G.$ If we apply it, the expression of the trace
simplifies to $(-u_{3}+u_{4})^{2}$ which is close to uniform, because squaring
in characteristic 2 is a permutation, and $u_4-u_3$ is approximately uniform.

As a next step we count the solutions with $u_{1}=0.$ In this case the trace
plus $s$ is
\[
vu_{3}u_{4}-u_{3}^{2}+u_{4}^{2}-vwu_{2}u_{3}-u_{2}^{2}-wu_{2}u_{4}+s.
\]

The equation $u_{1}u_{4}-u_{2}u_{3}=1$ gives us that $u_{3}=-1/u_{2}$. For any
choice of $u_{2}$, the above becomes a univariate polynomial in $u_{4}$ which is
non-zero because of the $u_{4}^{2}$ term. Hence the total number of solutions
with $u_{1}=0$ is $O(q).$ This amount does not affect the result, so from now on
we count the solutions with $u_{1}\ne0.$

We can now eliminate $u_{4}=(1+u_{2}u_{3})/u_{1}$ in $f'$. Renaming $u_{1}$,
$u_{2},$ and $u_{3}$ as $x,y,z$, respectively, we get the expression
\[
f':=vz(1+yz)/x-z^{2}+(1+yz)^{2}/x^{2}-vxy+x{}^{2}-vwyz+wxz-y^{2}-wy(1+yz)/x.
\]

First we note an upper bound of $O(q^{2})$ on the number of solutions to $f'=s$,
for any $s$. Indeed, after we pick $x$ and $y$ we are left with a quadratic
polynomial in $z$ which is not zero because of the $z^{2}$ term. Hence, this
polynomial has at most two solutions.

Next we show the stronger bound for all but $O(1)$ values of $s$. Letting
$f(x,y,z):=x^{2}(f'+s)$ and expanding and rearranging, we get the expression
\begin{multline*}
f:=x^{4}-x^{2}y^{2}-x^{2}z^{2}+y^{2}z^{2}+2yz+1\\
+v(-x^{3}y+xz+xyz^{2})+w(-xy-xy^{2}z+x^{3}z)-vwx^{2}yz+sx^{2}.
\end{multline*}

We shall show that if $f$ is not absolutely irreducible, then $s$ takes one of
$O(1)$ values. So if $s$ is not one of those values, then we can apply Theorem
\ref{le-lang-weil}. This will give the desired bound of $q^{2}+e_{s}$ on the
number of roots with $x,y,z\in F.$ We actually just wanted to count the roots
with $x\ne0.$ However, if $x=0$ then $f$ simplifies to $(1+yz)^{2}$ which has
$q-1$ roots. So the bound is correct even if we insist that $x\ne0.$

The function $f$ is a polynomial of degree 4 in three variables. Suppose that it
can be factorized as $f=PQ$. Note first that both $P$ and $Q$ must have a
constant term because $f$ has it. Also, neither $P$ nor $Q$ can have a power of
$y$ as a term, because $f$ does not have it (but such a term would arise in the
product between the highest-power such term in $P$ and in $Q$, one of which
could be the constant term). Similarly, neither can have a power of $z$ as a
term.

If $f=PQ$, then the sum of the degrees of $P$ and $Q$ is at most 4. If $P$ has
degree 3 then $Q$ has degree $1.$ By the above, $Q$ would be of the form $ax+b.$
However in this case there would be no way to produce the term $y^{2}z^{2}.$

So both $P$ and $Q$ have degree at most $2$, and we can write
\begin{gather*}
P=axy+byz+cxz+dx^{2}+ex+f,\\
Q=a'xy+b'yz+c'xz+d'x^{2}+e'x+f'.
\end{gather*}

Equating coefficients gives the systems of equations
\begin{align*}
xy^{2}z & \rightarrow ab'+a'b=-w\\
x^{2}yz & \rightarrow ac'+a'c+bd'+b'd=-vw\\
x^{3}y & \rightarrow ad'+a'd=-v\\
x^{2}y & \rightarrow ae'+a'e=0\\
xy & \rightarrow af'+a'f=-w\\
xyz^{2} & \rightarrow bc'+b'c=v\\
xyz & \rightarrow be'+b'e=0\\
yz & \rightarrow bf'+b'f=2\\
x^{3}z & \rightarrow cd'+c'd=w\\
x^{2}z & \rightarrow ce'+c'e=0\\
xz & \rightarrow cf'+c'f=v\\
x^{3} & \rightarrow de'+d'e=0\\
x^{2} & \rightarrow df'+f'd+ee'=s\\
x & \rightarrow ef'+e'f=0
\end{align*}

and

\begin{align*}
x^{2}y^{2} & \rightarrow aa'=-1\\
y^{2}z^{2} & \rightarrow bb'=1\\
x^{2}z^{2} & \rightarrow cc'=-1\\
x^{4} & \rightarrow dd'=1\\
1 & \rightarrow ff'=1.
\end{align*}

Multiplying by $bf$ the $yz$ equation and using that $bb'=ff'=1$, we find that
\[
b^{2}ff'+bb'f^{2}=b^{2}+f^{2}=2bf.
\]

Therefore, $(b-f)^{2}=0$ and so $b=f$. Since $bb'=ff'=1$, we also get that
$b'=f'$.

Now we claim that $e'=0.$ Assume for a contradiction that $e'\ne0.$ Multiplying
by appropriate variables, the equations with right-hand side equal to zero
become:
\begin{align*}
x^{2}y & \rightarrow a^{2}e'-e=0\\
xyz & \rightarrow b^{2}e'+e=0\\
x^{2}z & \rightarrow c^{2}e'-e=0\\
x^{3} & \rightarrow d^{2}e'+e=0.
\end{align*}
Summing the first two gives us that $(a^{2}+b^{2})e'=0$, which implies that
$a^{2}+b^{2}=0$ because $e'\ne0.$ Repeating this argument we obtain that
\[
a^{2}+b^{2}=a^{2}+d^{2}=b^{2}+c^{2}=c^{2}+d^{2}=0.
\]
Now multiplying the $xy^{2}z$ equation by $ab$ we get that
$a^{2}-b^{2}=2a^{2}=-wab$. Dividing by $ab\ne0$ we obtain that $2a/b=-w$.
Because $a^{2}/b^{2}=-1$, squaring we obtain that $w^{2}=-4$. Similarly,
multiplying the $x^{3}y$ equation by $ad$ we get that $a^{2}-d^{2}=2a^{2}=vad$
and we get that $v^{2}=-4$ as well. For odd $q$, this contradicts our assumption
that $(v^{2},w^{2})\ne(-4,-4)$. For even $q$ we have $4=0$ and so $v=w=0$ which
we were also excluding. Therefore $e'=0$. (From the equation for $xyz$ we get
that $e=0$ as well, but we will not use this.)

We can now simplify some of the equations as follows:
\begin{align*}
x^{2}yz & \rightarrow ac'+a'c+s=-vw\\
x^{2} & \rightarrow db'+d'b=s.
\end{align*}

Next, we handle the case of even $q$ where exactly one of $v$ or $w$ is $0$. If
$w=0$, then multiplying the $xy^{2}z$ equation by $ab$ we find that
$a^{2}-b^{2}=0$. So $a=b$ and the $x^{3}y$ equation has the same left-hand side
as the $x^{2}$ equation, which implies that $s=v.$ Similarly, if $v=0$, then the
$x^{3}y$ equation gives us that $a=d.$ Now the $xy^{2}z$ and the $x^{2}$
equation have the same left-hand side, giving us that $s=w.$

Now we continue the analysis for any $q$. Multiplying equations by appropriate
quantities we get:
\begin{align*}
xy^{2}z & \rightarrow a^{2}-b^{2}=-wab\\
x^{3}y & \rightarrow a^{2}-d^{2}=-vad\\
xyz^{2} & \rightarrow-b^{2}+c^{2}=vbc\\
x^{3}z & \rightarrow c^{2}-d^{2}=wcd.
\end{align*}

The first minus the second gives $-b^{2}+d^{2}=a(vd-wb)$; the third minus the
fourth gives $-b^{2}+d^{2}=c(vb-wd)$. And so
\[
a(vd-wb)=c(vb-wd).
\]

Now assume that $vd-wb\ne0$. Then by dividing by it and by $c\ne0$ we get
\[
\frac{a}{c}=\frac{vb-wd}{vd-wb}.
\]

So we have that
\begin{multline*}
\frac{a}{c}+\frac{c}{a}=\frac{(vb-wd)^{2}+(vd-wb)^{2}}{(vd-wb)(vb-wd)}=\frac{(b^{2}+d^{2})(v^{2}+w^{2})-4vwbd}{-vw(b^{2}+d^{2})+(w^{2}+v^{2})bd}\\
=\frac{(b^{2}+d^{2})(v^{2}+w^{2}-4vw/s)}{(b^{2}+d^{2})(-vw+(w^{2}+v^{2})/s)}=\frac{s(v^{2}+w^{2})-4vw}{-svw+w^{2}+v^{2}}.
\end{multline*}

Here we used the $x^{2}$ equation multiplied by $bd$, which is
$bds=b^{2}+d^{2},$ and then divided by $s$. So we are assuming that $s\ne0.$

Now if we plug this expression into the $x^{2}yz$ equation, which, using the
fact that $aa'=cc'=-1$, can be transformed into the equation $-a/c-c/a+s=-vw$,
we obtain that
\[
\frac{s(v^{2}+w^{2})-4vw}{-svw+w^{2}+v^{2}}+s=-vw.
\]

This expression can be satisfied by only a constant number of $s.$ Indeed,
taking the right-hand side to the left and multiplying by the denominator we
obtain the equation
\[
2s(v^{2}+w^{2})-4vw-s^{2}vw-sv^{2}w^{2}+vw(w^{2}+v^{2})=0.
\]

Now, if $q$ is odd and if exactly one of $v$ and $w$ is $0$ then all the terms
vanish except the first one, yielding that $s=0$. Together with our assumptions
and previous analysis, we can assume at this point that $vw\ne0$. In this case we obtain a
quadratic polynomial in $s$ which is not zero because of the $-s^{2}vw$ term.
This polynomial has at most two roots.

The case we left out is when $vd-wb=0$. In that case $d=bw/v.$ From the $x^{2}$
equation and the fact that $bb'=dd'=1$ we get that
\[
v/w+w/v=s.
\]

Altogether, we have shown that if the polynomial is not irreducible then $s$
takes one of at most six possible values. These values are $0,v,w,v/w+w/v,$ and
the at most two roots of the quadratic polynomial above. Although it does not
affect the result, we recall that these values of $s$ correspond to values of
$-s$ for the traces.

\section{Miscellaneous results} \label{s-misc}

In this section we first give an alternative proof of Theorem
\ref{t-gowers-bnp} for \sl2q.  Then we prove Theorem \ref{t-XYZ-cc}.

An immediate consequence of Theorem \ref{th-main-v2} with $t=2$ is that
the group SL$(2,q)$ has the property that the product of any four dense
sets is almost uniformly distributed. More precisely, we have the following
result.

\begin{theorem} \label{gisquasirandom}
Let $G$ be the group SL$(2,q)$, and let $A,B,C,D\subseteq G$ be
subsets of density $\a,\b,\g$ and $\d$, respectively. Then for every $g\in
G$,
\[|\E_{abcd=g}A(a)B(b)C(c)D(d)-\a\b\g\d|=O(|G|^{-c})\]
and
\[|\P[abcd=g|a\in A, b\in B, c\in C, d\in D]-|G|^{-1}|=(\a\b\g\d)^{-1}O(|G|^{-(1+c)}).\]
\end{theorem}


It turns out that from this result for four sets follows the same result for
three sets.

\begin{corollary} \label{threesets} Let $G$ be the group SL$(2,q)$, and let $A,B,C\subset G$ be subsets of density $\a,\b$ and $\g$, respectively. Then for every $g\in G$,
\[|\E_{abc=g}A(a)B(b)C(c)-\a\b\g|=O(|G|^{-c})\]
and
\[|\P[abc=g|a\in A, b\in B, c\in C]-|G|^{-1}|=(\a\b\g)^{-1}O(|G|^{-(1+c)}.\]
\end{corollary}

\begin{proof}
For each $a$, let $f(a)=A(a)-\a$. Then
\[\E_{abc=g}A(a)B(b)C(c)&=\a\E_{abc=g}B(b)C(c)+\E_{abc=g}f(a)B(b)C(c)\\
&=\a\b\g+\E_{abc=g}f(a)B(b)C(c).\\
\]
But
\[(\E_{abc=g}f(a)B(b)C(c))^2&\leq(\E_cC(c)^2)(\E_c(\E_{ab=gc^{-1}}f(a)B(b))^2)\\
&=\g\E_c\E_{ab=a'b'=gc^{-1}}f(a)B(b)f(a')B(b')\\
&=\g\E_{abb'^{-1}a^{-1}=e}(A(a)-\a)B(b)B(b')(A(a')-\a).\\
\]
There are four terms that make up the expectation. Each term that
involves at least one $\a$ is equal to $\pm\a^2\b^2$, with two minus signs
and one plus sign. The remaining term is $\a^2\b^2+O(|G|^{-c})$, by
Theorem \ref{gisquasirandom}. The first statement follows. Once again,
the second statement is equivalent to it by a simple application of Bayes's
theorem, together with the observation that
$\E_{abc=g}A(a)B(b)C(c)=\P[a\in A, b\in B, c\in C\,|\,abc=g]$.
\end{proof}

\begin{proof}[Proof of Theorem \ref{t-XYZ-cc}]
It suffices to prove the theorem in the case where $g$ is the identity $e$.
Let $a'$, $b'$, and $c'$ be selected independently and uniformly from $S$,
$G$, and $G$, respectively.

By Bayes' rule we can write $\Pr[abc=e] = \Pr[a'b'c'=e|a' \in A, b' \in B, c'
\in C] =\Pr[a' \in A, b' \in B, c' \in C | a'b'c'=e] \cdot \frac1{\alpha \beta
\gamma |G|}$.

Thus our goal is to show
\[|\Pr[a' \in A, b' \in B, c' \in C | a'b'c'=e] - \alpha \beta \gamma| \le |G|^{-\Omega(1)}.\]

Rewrite the difference as
\[
|\E_{a\in S,b \in G}A(a) B(b)D(a,b)|
\]
where $D(a,b)=C^{-1}(ab)-\gamma$.  Thinking of $D$ as a function defined on $S\times G$,
we can use Lemma \ref{boxnorminequality} to bound the fourth power of this expression by
\[\|A\|_{L_2}^4\|B\|_{L_2}^4\|D\|_\square^4\leq\|D\|_\square^4=\E_{b,b'}\E_{a,a'\in S}D(a,b)D(a,b')D(a',b)D(a',b').\]
If we change variables by premultiplying $b$ by $a^{-1}$ and $b'$ by $a'^{-1}$, then we
can rewrite the right-hand side as
\[
\E_{b,b'}\E_{a,a'\in S}D(1,b)D(a,a'^{-1}b')D(a',a^{-1}b)D(1,b').
\]

We claim that the quadruple $(b,b',aa'^{-1}b',a'a^{-1}b)$ is $\epsilon$-close
in statistical distance to $(v,w,x,wx^{-1}v)$, for $\epsilon \le
1/|G|^{\Omega(1)}$, where $v$, $w$, and $x$ are uniform in $G$.  It
suffices to show that the first three coordinates are jointly $\epsilon$-close
to uniform. But the distance is at most that of $aa'^{-1}$ from uniform.
It therefore follows by Lemmas \ref{l-sl2q} and \ref{lemma-cc-equidistribution}.(2) that
except for $O(1)$ conjugacy classes $S$, the first three coordinates are indeed $\eps$-close
to uniform.

Hence the value of the expression is at most $\epsilon$ plus
\[
\E_{v,w,x}(C^{-1}(v)-\gamma)(C^{-1}(w)-\gamma)(C^{-1}(x)-\gamma)(C^{-1}(wx^{-1}v)-\gamma).
\]

If we expand the product, in any term with at most three copies of $C^{-1}$
we can replace those copies by $\gamma$. Hence, by direct calculation or
say the binomial theorem, we can rewrite it as
\[
\E_{v,w,x}C^{-1}(v)C^{-1}(w)C^{-1}(x)C^{-1}(wx^{-1}v)-\gamma^{4}.
\]

By suitably redefining the copies of $C$ we can put this in the form of
Theorem \ref{gisquasirandom} and thereby obtain a bound of $|G|^{-\Omega(1)}$.
\end{proof}

\paragraph{Acknowledgments.}

We thank the anonymous referees for their useful feedback.  We also
thank Laci Pyber for pointing out Theorem 2.5 in \cite{Shalev08}, which we
used in \cite{GowersV-cc-int}.  Emanuele Viola is very grateful to Eric
Miles for extensive discussions during the early stages of this project. He
also thanks Laci Babai for an email exchange, and Swastik Kopparty for
pointing out the book \cite{Schmidt2004equations}.

\bibliographystyle{alpha}
{\small{ \ifnum\EmanueleViolaDir=1
\bibliography{../OmniBib}
\else
\bibliography{OmniBib}
\fi } }

\end{document}